\documentclass[11pt, a4paper]{article}

\usepackage[utf8]{inputenc}
\usepackage[T1]{fontenc}
\usepackage{amsmath, amssymb, amsthm}
\usepackage{geometry}
\usepackage{hyperref}
\usepackage{cite}

\newtheorem{theorem}{Theorem}
\newtheorem{conjecture}{Conjecture}
\newtheorem{proposition}{Proposition}

\title{A Counterexample to the Liu--Luo--Luo Conjecture on the Harmonic Landau Radius}
\author{Borovikov Mikhail\thanks{Lomonosov Moscow State University, Faculty of Mechanics and Mathematics, Leninskie Gory, 119991 Moscow, Russia. \texttt{Email: misha.borovikov@gmail.com}}}
\date{\today}

\begin{document}

\maketitle

\begin{abstract}
We disprove a conjecture of Liu, Luo, and Luo (2020) concerning the Landau radius for bounded planar harmonic mappings. We construct an explicit harmonic mapping $F_M$ satisfying the standard normalization conditions, which fails to be locally univalent strictly within the classical holomorphic Landau disk for all $M > M^* \approx 4.451$. By combining this counterexample with the lower bound established by Ponnusamy, Kalaj, and Vuorinen (2014), we demonstrate that the exact asymptotics of the harmonic Landau radius is $\frac{\pi}{8M}$.
\end{abstract}


\section{Introduction}

The classical Landau theorem states that if $f$ is a holomorphic function in the unit disk $\mathbb{D} = \{z \in \mathbb{C} : |z| < 1\}$ satisfying $f(0) = 0$, $|f'(0)| = 1$, and $|f(z)| < M$ for all $z \in \mathbb{D}$, then $f$ is univalent in the disk $|z| < r_0(M)$, where the holomorphic Landau radius is given by
\begin{equation}
r_0(M) = \frac{1}{M + \sqrt{M^2 - 1}}.
\end{equation}
This result is sharp, with the extremal function explicitly given by $f_0(z) = M z \frac{1 - Mz}{M - z}$.

The problem of extending Landau's theorem to the class of planar harmonic mappings was first posed by Chen, Gauthier, and Hengartner \cite{Chen2000}, who established the first versions of the theorem for bounded harmonic mappings. In 2009, Liu \cite{Liu2009} obtained improved Landau-type results for bounded harmonic mappings.

For a complex-valued harmonic mapping $F = h + \bar{g}$ in $\mathbb{D}$, where $h$ and $g$ are analytic, we define the quantity
\begin{equation}
\lambda_F(z) = |F_z(z)| - |F_{\bar{z}}(z)| = |h'(z)| - |g'(z)|.
\end{equation}
Under the conditions $F(0)=0$, $\lambda_F(0)=1$, and $|F(z)|<M$, Liu \cite{Liu2009} proved that $F$ is univalent in the disk $|z| < r_L(M)$, where
\begin{equation}
r_L(M) = 1 - \sqrt{\frac{\sqrt{2M^2 - 2}}{1 + \sqrt{2M^2 - 2}}}.
\end{equation}
This result is sharp when $M = 1$.

An improved lower bound for the radius of univalence of bounded harmonic mappings was established by Ponnusamy, Kalaj, and Vuorinen \cite{Kalaj2014}. They proved that under the conditions $F(0)=0$, $\lambda_F(0)=1$, and $|F(z)|<M$, the mapping $F$ is univalent in the disk $|z| < r_{PKV}$, where
\begin{equation}
r_{PKV} = 1 - \sqrt{\frac{4M}{4M+\pi}}.
\end{equation}
For $M > \frac{\pi}{\sqrt{\pi^2-8}} \approx 2.298$, this estimate improves the bound of Liu \cite{Liu2009}.

In a recent paper, Liu, Luo, and Luo \cite{Liu2020} considered the class of \emph{strongly bounded} harmonic mappings, that is, mappings $F = h + \bar{g}$ satisfying the stronger condition $|h(z)| + |g(z)| < M$. They proved that for such mappings, the Landau radius coincides exactly with the classical holomorphic one:

\begin{theorem}[Liu--Luo--Luo, 2020]
Suppose that $M>1$, $F(z)=h(z)+\overline{g(z)}$ is harmonic in $\mathbb{D}$ with $F(0)=0$, $\lambda_F(0)=1$, and $|h(z)| + |g(z)| < M$ for all $z \in \mathbb{D}$. Then $F(z)$ is univalent on $\mathbb{D}_{r_0}$, where $r_0 = \frac{1}{M + \sqrt{M^2 - 1}}$.
\end{theorem}

Motivated by this result, they posed the following conjecture for bounded harmonic mappings:

\begin{conjecture}[Liu--Luo--Luo, 2020]
Suppose that $M>1$, $F(z)=h(z)+\overline{g(z)}$ is harmonic in $\mathbb{D}$ with $F(0)=0$, $\lambda_F(0)=1$, and $|F(z)| < M$ for all $z \in \mathbb{D}$. Then $F(z)$ is univalent on $\mathbb{D}_{r_0}$, where $r_0 = \frac{1}{M + \sqrt{M^2 - 1}}$.
\end{conjecture}

In this note, we disprove this conjecture for all sufficiently large $M$ by constructing an explicit counterexample.

\section{Construction of the Counterexample}

Define the real-valued harmonic function
\begin{equation}
u(z) = \frac{2}{\pi} \text{Arg} \left( \frac{1 - z^2}{1 + z^2} \right),
\end{equation}
where the branch of the argument is chosen such that $-\frac{\pi}{2} < \text{Arg} \left( \frac{1 - z^2}{1 + z^2} \right) < \frac{\pi}{2}$. Since the mapping $z \mapsto -z^2$ maps $\mathbb{D}$ into $\mathbb{D}$, the function $\frac{1 - z^2}{1 + z^2}$ has positive real part in $\mathbb{D}$. Consequently, $u$ is harmonic in $\mathbb{D}$, satisfying $|u(z)| < 1$ and $u(0) = 0$. In real coordinates $z = x + iy$, the function takes the form
\begin{equation}
u(x, y) = -\frac{2}{\pi} \arctan \left( \frac{4xy}{1 - (x^2 + y^2)^2} \right).
\end{equation}

For a given $M > 1$, we define the harmonic mapping
\begin{equation}
F_M(z) = z + i(M - 1)u(z).
\end{equation}

\begin{theorem}
Let
\begin{equation}
M^* = \frac{16 - 2\pi + \pi^{3/2}}{4(4-\pi)} \approx 4.451.
\end{equation}
For every $M > M^*$, the mapping $F_M$ satisfies the conditions
\begin{equation}
F_M(0) = 0, \quad \lambda_{F_M}(0) = 1, \quad |F_M(z)| < M \quad (z \in \mathbb{D}),
\end{equation}
but fails to be locally univalent at some point $r_M \in (0, r_0(M))$.
\end{theorem}

\begin{proof}
Near the origin, $u(x, y) = -\frac{8}{\pi} xy + O(|z|^4)$. Therefore, $u_x(0) = u_y(0) = 0$, which implies $(F_M)_z(0) = 1$ and $(F_M)_{\bar{z}}(0) = 0$, yielding $\lambda_{F_M}(0) = 1$. Furthermore, $F_M(0) = 0$. Since $|z| < 1$ and $|u(z)| < 1$ for all $z \in \mathbb{D}$, we obtain the strict bound $|F_M(z)| \leq |z| + (M - 1)|u(z)| < M$. Thus, $F_M$ strictly satisfies all normalization and boundedness conditions.

In real coordinates, $F_M(x + iy) = x + i\big(y + (M - 1)u(x, y)\big)$. The Jacobian determinant of $F_M$ is
\begin{equation}
J_{F_M}(x, y) = 1 + (M - 1)u_y(x, y).
\end{equation}
On the real axis ($y = 0$), we have $u(x, 0) = 0$. Differentiating the expression for $u$ by $y$, we find
\begin{equation}
u_y(x, 0) = -\frac{8x}{\pi(1 - x^4)}, \quad 0 < x < 1.
\end{equation}
Hence, the Jacobian $J_{F_M}$ on the real axis is
\begin{equation}
J_{F_M}(x, 0) = 1 - \frac{8(M - 1)}{\pi} \frac{x}{1 - x^4}.
\end{equation}
This function vanishes at a unique point $r_M \in (0, 1)$ satisfying
\begin{equation}
8(M - 1)r_M = \pi(1 - r_M^4).
\end{equation}
By Lewy's theorem \cite{Lewy1936}, a harmonic mapping is locally univalent if and only if its Jacobian is non-vanishing. Therefore, $F_M$ is not locally univalent at the point $r_M$.

From the defining equation for $r_M$, it follows immediately that $r_M < \frac{\pi}{8(M - 1)}$. On the other hand, the holomorphic Landau radius is $r_0(M) = \frac{1}{M + \sqrt{M^2 - 1}}$. Therefore, $r_M < r_0(M)$ holds whenever
\begin{equation}
\frac{\pi}{8(M - 1)} < \frac{1}{M + \sqrt{M^2 - 1}}.
\end{equation}
For all $M > M^*$, this inequality holds, which gives $r_M < r_0(M)$.
\end{proof}

\section{Asymptotics}
We now analyze the asymptotic behavior of the point of non-univalence $r_M$ as $M \to \infty$. From the defining equation $8(M - 1)r_M = \pi(1 - r_M^4)$, it follows that $r_M \to 0$ as $M \to \infty$. Therefore, $r_M^4 \to 0$, and we obtain the asymptotic behavior:
\begin{equation}
r_M \sim \frac{\pi}{8M} \quad \text{as } M \to \infty.
\end{equation}

This asymptotic behavior coincides exactly with the asymptotics of the Ponnusamy--Kalaj--Vuorinen radius $r_{PKV}$ established in \cite{Kalaj2014}. Indeed, expanding their result for large $M$ yields:
\begin{equation}
r_{PKV} = 1 - \left(1 + \frac{\pi}{4M}\right)^{-1/2} \sim \frac{\pi}{8M} \quad \text{as } M \to \infty.
\end{equation}

Let $r_{\text{harm}}(M)$ denote the harmonic Landau radius, that is, the supremum of all $r > 0$ such that every harmonic mapping $F$ satisfying $F(0)=0$, $\lambda_F(0)=1$, and $|F(z)|<M$ is univalent in $|z|<r$. By definition, $r_{\text{harm}}(M)$ is bounded below by $r_{PKV}$ and bounded above by $r_M$. Since both bounds have the same asymptotics $\frac{\pi}{8M}$, we obtain the following observation:

\begin{proposition}
The exact asymptotics of the harmonic Landau radius is
\begin{equation}
r_{\text{harm}}(M) \sim \frac{\pi}{8M} \quad \text{as } M \to \infty.
\end{equation}
\end{proposition}

\end{document}